\documentclass[12]{amsart}
\usepackage{amsfonts,amssymb,amscd}
\usepackage{amsxtra}
\usepackage{epsfig,amssymb}
\usepackage{mathrsfs}
\usepackage[mathscr]{eucal}
\usepackage{epsfig}
\usepackage{bm} 
\usepackage{ulem} 
\usepackage{hyperref}

\newcommand{\Z}{{\mathbb Z}}
\newcommand{\Q}{{\mathbb Q}}

\newcommand{\spec}{\mathrm{ Spec}\,}

\def\be{\kern -.1em}
\def\le{\kern 0.03em}
\def\lbe{\kern -.025em}
\def\e{\kern 0.08em}

 \newcommand{\ksep}{k^{\le\rm s}}
\newcommand{\kbar}{\overline{k}}

\newcommand{\mr}{\mathrm }
\newcommand{\mm}{\mathfrak m}

\newcommand{\fppf}{\mathrm{fppf}}

\newcommand{\Hom}{ \mathrm{Hom}}

\numberwithin{equation}{section}

\newtheorem{theorem}{Theorem}[section]

\newtheorem{proposition}[equation]{Proposition}
\newtheorem{corollary}[equation]{Corollary}

\theoremstyle{definition}

\theoremstyle{remark}
\newtheorem{remark}[equation]{Remark}

\begin{document}

\title{Groups of Components and Weil Restriction}
  
 \author{Alessandra Bertapelle}
\address{Universit\`a degli Studi di Padova, Dipartimento di Matematica, via Trieste 63, I-35121 Padova}
\email{alessandra.bertapelle@unipd.it}
\thanks{A. B. was partially supported by  PRAT 2013 ''Arithmetic of Varieties over Number Fields" CPDA 135371/13}

\author{Cristian D. Gonz\'alez-Avil\'es}
\address{Departamento de Matem\'aticas, Universidad de La Serena, Cisternas 1200, La Serena 1700000, Chile}
\email{cgonzalez@userena.cl}
\thanks{C.\,G.-A. was partially supported by Fondecyt grant 1120003.}
\date{\today}   

\begin{abstract} We determine the behavior under Weil restriction of the group of connected components of the special fiber of an arbitrary smooth group scheme (whose Weil restriction exists) over an arbitrary (commutative and unital) local ring. Applications to N\'eron models are given.
\end{abstract}

\maketitle 

\section{Introduction}

Let $R$ be a discrete valuation ring with fraction field $K$ and residue field $k$ and let $\ksep$ be a fixed separable algebraic closure of $k$. Let $A_{K}$ be an abelian variety over $K$ with N\'eron model $A$ over $R$ and let $\pi_{0}(\lbe A_{\le\rm s}\lbe)$ denote the $k$-group scheme of connected components of the special fiber $A_{\le\rm s}$ of $A$.  Let $A^{\vee}_{K}$ be the dual (i.e., Picard) variety of $A_{K}$ and write $A^{\vee}$ and $\pi_{0}(\lbe A^{\vee}_{\le\rm s}\lbe)$ for the corresponding objects associated to the abelian variety $A^{\vee}_{K}$. 
When $k$ is perfect, Grothendieck's pairing $\pi_{0}(\lbe A_{\le\rm s}\lbe )(\ksep)\times\pi_{0}(\lbe A^{\vee}_{\le\rm s})(\ksep)\to\Q/\Z$ \cite[IX, (1.2.2)]{sga7} has been verified to be perfect in all cases except when $K$ has positive characteristic and $k$ is infinite; see \cite[\S 1]{bb} for further comments. More recently  a  general proof of  the perfectness of Grothendieck's pairing has been announced in \cite{s}.   However, when $k$ is {\it imperfect}, there exist examples that show that the above pairing is no longer perfect. The first such examples were constructed in 
\cite[comment after Corollary 2.5]{bb} using the Weil restriction functor. It is for this reason, at least, that a full understanding of the behavior of the groups of connected components of abelian (or more general smooth group) varieties under Weil restriction is desirable.  Previously, this has been achieved only in the case of abelian varieties under various restrictions. See Remark \ref{rem}. In this paper we completely determine the behavior under Weil restriction of the group of connected components of the special fiber of an arbitrary (possibly non commutative) smooth group scheme over an arbitrary (possibly non noetherian) local ring. Our main result is the following

\begin{theorem}\label{main}
Let  $R^{\e\prime}/R$ be a finite flat extension of local rings with associated residue field extension $k^{\le\prime}\be/k$. Let $G^{\le\prime}$ be a smooth $R^{\e\prime}$-group scheme such that the Weil restriction $\Re_{\le R^{\e\prime}\be/\be R}(G^{\e\prime}\le)$ exists. Then $\Re_{\le k^{\le\prime}\be/k}(\pi_{0}(G^{\e\prime}_{\lbe{\rm{s}}}\le))$ exists and there exists a canonical isomorphism of \'etale $k$-group schemes
\[
\pi_{0}(\Re_{\le R^{\e\prime}\be/\be R}(G^{\e\prime}\le)_{{\rm{s}}})\simeq\Re_{\le k^{\le\prime}\be/k}(\pi_{0}(G^{\e\prime}_{\lbe{\rm{s}}}\le)).
\]
\end{theorem}
The theorem applies, in particular, to smooth and quasi-projective $R^{\e\prime}$-group schemes since $\Re_{\le R^{\e\prime}\be/\be R}(G^{\e\prime}\le)$ is well-known to exist in this case. More general conditions that guarantee the existence of $\Re_{\le R^{\e\prime}\be/\be R}(G^{\e\prime}\le)$  can be found in \cite[\S7.6, Theorem 4, p.~194]{blr} and \cite[Theorem 2.15]{bga}.

The proof of the theorem, which is simpler than the proofs of the particular cases alluded to above, essentially relies only on basic properties of the Weil restriction functor and the following key fact: if  $k$ is a field, $R^{\e\prime}$ is a finite local $k$-algebra and $G^{\e\prime}$ is a smooth   $R^{\e\prime}$-group scheme with geometrically connected special fiber, then $\Re_{\le R^{\e\prime}\be/\lbe k}(G^{\e\prime}\le)$ is geometrically connected.

 Some applications of Theorem \ref{main} are discussed at the end of the paper. 

\medskip

We thank  Brian Conrad for helping us shorten our original proof of the main theorem. We also thank Bas Edixhoven and Dino Lorenzini for helpful comments.

\section{Preliminaries}

All rings considered in this paper are commutative and unital. If $X$ is a scheme, the topological space underlying $X$ will be denoted by $|X|$. The identity morphism of $X$ will be denoted by $1_{\be X}$.

Let $S$ be a scheme, let $X$ be an $S$-scheme and let $S^{\e\prime\prime}\to S^{\le\prime}\to S$ be morphisms of schemes. We will make the identification
\begin{equation}\label{id1}
(X\times_{S}S^{\e\prime}\e)\times_{S^{\le\prime}}S^{\e\prime\prime}=X\times_{S}S^{\e\prime\prime}.
\end{equation}
If $A\to B$ is a ring homomorphism and $X$ is an $A$-scheme,  
$X_{\be B}$ will denote $X\times_{\spec A}\spec B$.

Let $W\to S^{\e\prime}\to S$ be morphisms of schemes. If $Z$ is an $S$-scheme and $\textrm{pr}_{\le 1}\colon Z\be\times_{S}\lbe S^{\e\prime}\to Z$ is the first projection then, by the universal property of the fiber product, the map
\begin{equation}\label{bc}
\Hom_{S^{\prime}}(W,Z\be\times_{S}\be S^{\e\prime}\e)\overset{\be\!\sim}{\to}\Hom_{\le S}(W,Z\le),\quad g\mapsto \textrm{pr}_{1}\circ g,
\end{equation}
is a bijection.

Let $A$ be an artinian local ring with residue field $k$ and let $G$ be an $A$-group scheme {\it locally of finite type}. We will identify $G$ with the functor on $(\mr{Sch}/A)$ represented by $G$. Let $G_{\lbe \textrm{s}}:=G\!\times_{\spec A}\be\spec k$ be the special fiber of $G\e$\footnote{\label{foot} Here the subscript ``$\rm s$" indicates ``special fiber'' rather than the closed point of $\spec A$. We use this notation because it remains unchanged by a change of rings, i.e., if $A^{\prime}$ is another artinian local ring with residue field $k^{\e\prime}$ and $G^{\e\prime}$ is an $A^{\prime}$-group scheme, then it should be clear that $G^{\e\prime}_{\lbe \textrm{s}}$ means $G^{\e\prime}\!\times_{\spec A^{\prime}}\be\spec k^{\e\prime}$.}. Now let $|G|^{\le 0}\subset |G|$ denote the connected component of the identity $1\in G$ and define a subfunctor $G^{\e 0}$ of $G$ by
\begin{equation}\label{idcomp1}
G^{\e 0}(T)=\{u\in G(T)\colon u(|T|)\subseteq |G|^{\le 0}\},
\end{equation}
where $T$ is any $A$-scheme. Since $G$ is locally of finite type over $A$, \eqref{idcomp1} is represented by an open and normal subgroup scheme $G^{\e 0}$ of $G$ \cite[$\text{VI}_{\text{A}}$, Proposition 2.3.1]{sga3}. Note also that, since every subgroup scheme of $G$ is closed \cite[$\text{VI}_{\text{A}}$, Corollary 0.5.2]{sga3}, $G^{\e 0}$ is open and closed in $G$. Assume now that $G$ is locally of finite type and {\it flat} over $A$. Then, by \cite[$\text{VI}_{\text{A}}$, Proposition 5.5.1(i) and Theorem 3.2(vi)]{sga3}, the  fppf quotient sheaf of groups 
\[
\pi_{0}(G):=G^{\e 0}\e\backslash G 
\]
is represented by an \'etale $A$-group scheme and the canonical morphism 
\begin{equation}\label{pim}
p_{\le G}\colon G\to \pi_{0}(G)
\end{equation}
is faithfully flat and locally of finite presentation. Further, by \cite[$\text{VI}_{\text{B}}$, Proposition 3.3]{sga3}, there exists a canonical isomorphism of $k$-group schemes $(G_{\rm s}\lbe)^{0}\simeq (G^{\e 0})_{\rm s}$ and therefore $\pi_{0}(G_{\rm s}\le)\simeq\pi_{0}(G\le)_{\rm s}$.

Note that, if $G$ is {\it smooth} over $A$, then $G^{\e 0}$ is smooth as well and therefore \eqref{pim} is a smooth morphism by \cite[$\text{IV}_{4}$, Proposition 17.5.1]{ega} and \cite[${\rm{VI}}_{\rm B}$, Proposition 1.3]{sga3}. 

\medskip

Let $f\colon S^{\e\prime}\to S$ be a morphism of schemes and let $X^{\prime}$ be an $S^{\le\prime}$-scheme. We will say that {\it the Weil restriction of $X^{\prime}$ along $f$ exists}  (or, more concisely, that {\it $\Re_{S^{\le\prime}\be/S}(X^{\prime})$ exists}) if the contravariant functor
$(\mathrm{Sch}/S)\to(\mathrm{Sets}), T\mapsto\Hom_{
S^{\le\prime}}(T\times_{S}S^{\le\prime},X^{\le\prime}\le)$, is  representable, i.e., if there exists a pair $(\Re_{S^{\le\prime}\be/S}(X^{\prime}\e), q_{\le X^{\prime}\!,\e S^{\le\prime}\be/S})$, where $\Re_{S^{\le\prime}\be/S}(X^{\prime}\e)$ is an $S$-scheme and  $q_{\le X^{\prime}\be,\e S^{\le\prime}\be/S}\colon \Re_{S^{\le\prime}\be/S}(X^{\prime}\e)_{S^{\le\prime}}\to X^{\prime}$ is an $S^{\prime}$-morphism of schemes, such that the map
\begin{equation}\label{wr}
\Hom_{\le S}\e(T,\Re_{S^{\le\prime}\be/S}(X^{\le\prime}\e))\overset{\!\sim}{\to}\Hom_{
S^{\le\prime}}(T\!\times_{S}\!S^{\e\prime},X^{\le\prime}\e), \quad g\mapsto q_{\le X^{\prime}\be,\e S^{\le\prime}\be/S}\circ g_{\le S^{\le\prime}} 
\end{equation} is a bijection \cite[(1.1.8), p.~22]{ega1}. The pair $(\Re_{S^{\le\prime}\be/S}(X^{\prime}\e), q_{\le X^{\prime}\!,\e S^{\le\prime}\be/S})$ (or, more concisely, the scheme $\Re_{S^{\le\prime}\be/S}(X^{\prime}\e)$) is called the {\it Weil restriction of $X^{\prime}$ along $f$}. If $S^{\le\prime}=\spec A^{\le\prime}$ and $S=\spec A$ are affine, we will write $(\Re_{A^{\le\prime}/A}(X^{\prime}\le),q_{\le X^{\prime}\!,\e A^{\le\prime}\lbe/A})$ for $(\Re_{S^{\le\prime}\be/S}(X^{\prime}\le),q_{\le X^{\prime}\!,\e S^{\le\prime}\lbe/S})$ and refer to $\Re_{A^{\le\prime}/A}(X^{\prime}\le)$ as the {\it Weil restriction of $X^{\prime}$ along $A^{\le\prime}/A$}.

If $X^{\prime}$ is an $S^{\e\prime}$-scheme such that $\Re_{S^{\le\prime}\be/S}(X^{\prime}\le)$ exists and $T\to S$ is any morphism of schemes, then $\Re_{\le S^{\le\prime}_{T}\lbe/ T}\le(\e X^{\prime}\!\times_{S^{\prime}}\! S^{\e\prime}_{\le T})$ exists as well and \eqref{id1}, \eqref{bc} and \eqref{wr} yield a canonical isomorphism of $T$-schemes 
\begin{equation}\label{wrbc}
\Re_{S^{\le\prime}\be/S}(X^{\prime}\le)\times_{S}T\overset{\!\sim}{\to}\Re_{\le S^{\le\prime}_{T}\lbe/ T}\le(\e X^{\prime}\!\times_{S^{\prime}}\! S^{\e\prime}_{\le T}).
\end{equation}
Further, if $S^{\le\prime\prime}\to S^{\le\prime}\to S$ are morphisms of schemes and  $X^{\prime\prime}$ is an $S^{\le\prime\prime}$-scheme such that $\Re_{\e S^{\prime\prime}\be/S^{\le\prime}\be}\e(X^{\prime\prime})$ exists, then \eqref{wr} implies that $\Re_{S^{\prime}\be/S}(\Re_{\e S^{\prime\prime}\be/S^{\le\prime}\be}\e(X^{\prime\prime}))$ exists if, and only if,  $\Re_{\e S^{\prime\prime}\!/S} (X^{\prime\prime})$ exists. If this is the case, then there exists a canonical isomorphism of $S$-schemes 
\begin{equation}\label{wrcomp}
\Re_{S^{\prime}\be/S}(\Re_{\e S^{\prime\prime}\be/S^{\le\prime}\be}\e(X^{\prime\prime}))\overset{\!\lbe\sim}{\to}\Re_{\e S^{\prime\prime}\be/S} (X^{\prime\prime})  .
\end{equation}

Let $f\colon S^{\e\prime}\to S$ be a finite and locally free (i.e., flat and of finite presentation) morphism of schemes and let $X^{\e\prime}$ be an $S^{\e\prime}$-scheme. By \cite[\S7.6, Theorem 4, p.~194]{blr}, $\Re_{S^{\le\prime}\be/S}(X^{\prime}\e)$ exists if, for each point $s\in S$, every finite set of points of $X^{\le\prime}\times_{S}\spec k(s)$ is contained in an open affine subscheme of $X^{\prime}$. This condition is satisfied if $X^{\e\prime}$ is quasi-projective over $S^{\e\prime}$ \cite[II, Definition 5.3.1 and Corollary 4.5.4]{ega}. A weaker condition for the existence of $\Re_{S^{\le\prime}\be/S}(X^{\prime}\e)$ is given in \cite[Theorem 2.15]{bga}.

\begin{proposition}\label{et-eq} Let $k$ be a field and let $A$ be a finite local $k$-algebra with residue field $k^{\e\prime}$. Let $X$ be an \'etale $A$-scheme. Then the following hold.
\begin{itemize}
\item[(i)] $\Re_{k^{\le\prime}\be/A}(X_{\rm s})$ exists and is canonically isomorphic to $X$.
\item[(ii)] $\Re_{A/k}(X)$ and $\Re_{\le k^{\le\prime}\be/k}(X_{\rm s}\lbe)$ exist and are canonically isomorphic.
\end{itemize} 
\end{proposition} 
\begin{proof} Let $T$ be any $A$-scheme. Since $T_{\rm s}\to T$ is a nilpotent immersion and $X$ is formally \'etale over $A$, the canonical map  $\Hom_{A}\e(T,X\le)\to\Hom_{\e k^{\prime}}\e(T_{\rm s},X_{\rm s}\le),\, g\mapsto g_{\e\rm s},$ is a bijection \cite[$\text{IV}_{4}$, Remark 17.1.2(iv)]{ega}. Thus
$(\Re_{k^{\le\prime}\be/A}(X_{\rm s}), q_{\le X_{\rm s}\lbe,\e k^{\le\prime}\be/A})=(X,1_{\be X_{\rm s}})$ is the Weil restriction of $X_{\rm s}$ along $k^{\e\prime}/A$. Assertion (i) is now clear. Now, by (i), $\Re_{\le A\lbe/k}(X)$ exists if, and only if, $\Re_{\le A\lbe/k}(\Re_{\e k^{\prime}\be/A}(X_{\rm s}))\simeq \Re_{\e k^{\prime}\be/k}(X_{\rm s})$ \eqref{wrcomp} exists and, if this is the case, then $\Re_{\le A\lbe/k}(X)$ and  $\Re_{\le k^{\le\prime}\be/k}(X_{\rm s}\lbe)$ are canonically isomorphic \'etale $k$-schemes. Thus it remains to show that 
$\Re_{\e k^{\prime}\be/k}(X_{\rm s})$ exists. By \cite[\S7.6, Theorem 4, p.~194]{blr}, it suffices to check that  any finite set of points of $X_{\rm s}$ is contained is an affine open subscheme of $X_{\rm s}$. Since $X_{\rm s}$ is an \'etale $k^{\e\prime}$-scheme, $X_{\rm s}$ is isomorphic to a sum $\coprod_{i\in I} \spec k^{\e\prime}_{i}$, where each $k^{\e\prime}_{i}$ is a finite and separable extension of $k^{\e\prime}$ \cite[$\text{IV}_{4}$, Corollary 17.4.2$({\text d}^{\prime}\e)$]{ega}. It is now clear that the stated condition holds.  
\end{proof}

\section{Proof of the main theorem}
An {\it extension of  local rings} is a flat homomorphism $R\to R^{\e\prime}$ of  local rings. Then $R\to R^{\e\prime}$ is faithfully flat and therefore injective \cite[Chapter 2, (4.A), Corollary, p.~27, and (4.C)(i), p.~28]{mat}. Let $\mm$ and $\mm^{\e\prime}$ denote the maximal ideals of $R$ and $ R^{\e\prime}$, respectively. Then $R\to R^{\e\prime}$ induces an extension of residue fields $k=R/\mm\hookrightarrow  k^{\e\prime}=R^{\e\prime}/\mm^{\e\prime}$.

\begin{proposition}\label{op4} Let $k$ be a field and let $A$ be a finite local $k$-algebra with residue field $k^{\e\prime}$. Let $G$ be a smooth $A$-group scheme  such that $\Re_{\le A/k}(G\le)$ exists. Then $\Re_{\e k^{\le\prime}\be/k}(\pi_{0}(G_{{\rm{s}}}\le))$ exists and there exists a canonical isomorphism of \'etale $k$-group schemes $\pi_{0}(\Re_{\le A/k}(G\le))\simeq\Re_{\e k^{\le\prime}\be/k}(\pi_{0}(G_{{\rm{s}}}\le))$.
\end{proposition}
\begin{proof} The existence assertion follows from Proposition \ref{et-eq}(ii). Now, since $G^{\e 0}$ is an open and closed subgroup scheme of $G$, $\Re_{\le A/k}(G^{\e 0})$ is an open and closed subgroup scheme of $\Re_{\le A/k}(G\le)$ by \cite[\S7.6, Proposition 2, p.~192]{blr}. Further, since $(G^{\e 0})_{\rm s}$ is geometrically connected \cite[$\text{VI}_{\text{A}}$, Proposition 2.4(i)]{sga3}, $\Re_{\le A/k}(G^{\e 0})$ is geometrically connected as well by \cite[Proposition A.5.9]{cgp}\footnote{The proof in [loc.cit.] remains valid if the quasi-projectivity assumption made there is replaced by the condition that the Weil restrictions involved exist.}. It follows that $\Re_{\le A/k}(G^{\e 0})=\Re_{\le A/k}(G\le)^{0}$. Now, since $p_{\le G}\colon G\to\pi_{0}(G\le)$ is smooth and surjective, the induced smooth morphism $\Re_{\le A/k}(G\le)\to \Re_{\le A/k}(\pi_{0}(G\le))$ is surjective by \cite[Corollary A.5.4(1)]{cgp}. Thus there exists an exact and commutative diagram of sheaves of groups on $(\textrm{Sch}/k)_{\fppf}$
\[
\begin{CD}
1@>>>\Re_{\le A\be/k}(G\le)^{0}@>>>\Re_{\le A\be/k}(G\le)@>>>\pi_{0}(\Re_{\le A\be/k}(G\le)) @>>>1
\\
@. @| @| @VV\wr V\\
1@>>>\Re_{\le A\be/k}(G^{\e 0})@>>> \Re_{\le A/k}(G\le)@>>>
\Re_{\le A/k}(\pi_{0}(G\le))@>>>1,
\end{CD}
\]
which yields a canonical isomorphism of \'etale $k$-group schemes $ \pi_{0}(\Re_{\le A\be/k}(G\le))\simeq \Re_{\le A/k}(\pi_{0}(G\le))$. Now Proposition \ref{et-eq}(ii) completes the proof.
\end{proof}

{\it We may now prove Theorem \ref{main}.}

The existence assertion follows from Proposition \ref{et-eq}(ii). Let $A:=R^{\e\prime}\otimes_{R}k=R^{\e\prime}/\mm R^{\e\prime}$, which is a finite local $k$-algebra with residue field $k^{\e\prime}$ and let $S^\prime=\spec R^\prime$. Then $G^{\e\prime}_{\! A}:=G^{\e\prime}\!\times_{S^{\le\prime}}\be\spec A$ is a smooth $A$-group scheme. Now, since $\Re_{\le R^{\e\prime}\be/\be R}(G^{\e\prime}\le)$ exists by hypothesis, $\Re_{A/k}(G^{\e\prime}_{\!  A}\le)$ exists as well and the map \eqref{wrbc} for $(X^{\prime}\be,\lbe S^{\e\prime}\be/S,T\e)=(G^{\e\prime}\be,\lbe R^{\e\prime}\be/R, k\le)$ is an isomorphism of smooth $k$-group schemes $\Re_{\le R^{\e\prime}\be/\be R}(G^{\e\prime}\le)_{{\rm{s}}}\simeq
\Re_{A/k}(G^{\e\prime}_{\!  A}\le)$. On the other hand, Proposition \ref{op4} yields a canonical isomorphism of \'etale $k$-group schemes $\pi_{0}(\Re_{A/k}(G^{\e\prime}_{\! A})\e)\simeq\Re_{\le k^{\le\prime}\be/k}(\pi_{0}(G^{\e\prime}_{{\rm{s}}}))$, whence the theorem follows.
\medskip

\section{Some applications}

We present the following consequences of Theorem \ref{main}.

\begin{corollary} \label{cor2} Let $R^{\e\prime}\be/R$ be a finite extension of  local rings such that the associated extension of residue fields $k^{\e\prime}\be/k$ is {\rm purely inseparable}. Let $G^{\e\prime}$ be a smooth $R^{\e\prime}$-group scheme such that $\Re_{\le R^{\e\prime}\be/\be R}(G^{\e\prime}\le)$ exists. Then there exists a canonical isomorphism of \'etale $k^{\e\prime}$-group schemes $\pi_{0}(\Re_{\le R^{\e\prime}\be/\be R}(G^{\e\prime}\le)_{{\rm{s}}})_{k^{\le\prime}}\simeq\pi_{0}(G^{\e\prime}_{{\rm{s}}}\le)$.
\end{corollary}
\begin{proof} If $\chi\colon\pi_{0}(\Re_{\le R^{\e\prime}\be/\be R}(G^{\e\prime}\le)_{{\rm{s}}})\overset{\!\sim}{\to}\Re_{\le k^{\le\prime}\be/k}(\pi_{0}(G^{\e\prime}_{{\rm{s}}}\le))$ is the isomorphism of Theorem \ref{main}, then the diagram of \'etale $k^{\le\prime}$-group schemes
\[
\begin{CD}
\pi_{0}(\Re_{\le R^{\e\prime}\be/\be R}(G^{\e\prime}\le)_{{\rm{s}}})_{k^{\le\prime}} @>\chi_{ k^{\le\prime}}>>\Re_{\le k^{\le\prime}\be/k}(\pi_{0}(G^{\e\prime}_{{\rm{s}}}\le))_{k^{\le\prime}}\\
@VV\pi_{0}((q_{\le G^{\le\prime}\be,\le R^{\le\prime}\be/\be R}\lbe)_{k^{\le\prime}})V   @VV q_{\e\pi_{0}(G^{\le\prime}_{\rm{s}}\le)\lbe,\, k^{\le\prime}/k}V\\
\pi_{0}(G^{\e\prime}_{\textrm{s}}\le) @= \pi_{0}(G^{\e\prime}_{\textrm{s}}\le)
\end{CD}
\]
commutes. Further, since $\spec k^{\e\prime}\to\spec k$ is a finite and locally free universal homeomorphism, the right-hand vertical map in the above diagram is an isomorphism\footnote{Proofs of these facts, currently available at \url{http://arxiv.org/abs/1501.05621v3}\,, will appear in \cite{bga}}.
Thus the left-hand vertical map is an isomorphism as well, which completes the proof. 
\end{proof}

\begin{corollary}\label{cor1} Let $R^{\e\prime}\be/R$ be a finite extension of discrete valuation rings with associated residue and fraction field extensions $k^{\le\prime}\be/k$ and $K^{\le\prime}\be/K$. Let $G^{\e\prime}$ be a smooth, separated and commutative $R^{\e\prime}$-group scheme which is a N\'eron model\,\footnote{\,Only assumed to be locally of finite type \cite[\S 10.1. Definition 1, p.~289]{blr}.} of its generic fiber $G^{\e\prime}_{\lbe\eta}$. Then $\Re_{\le R^{\e\prime}\be/\be R}(G^{\e\prime}\le)$ exists, is a N\'eron model of $\Re_{\le K^{\prime}\be/\be K}\lbe(G^{\e\prime}_{\lbe\eta}\lbe)$ and there exists a canonical isomorphism of \'etale $k$-group schemes $\pi_{0}(\Re_{\le R^{\e\prime}\be/\be R}(G^{\e\prime}\le)_{{\rm{s}}})\simeq\Re_{\le k^{\le\prime}\be/k}(\pi_{0}(G^{\e\prime}_{\lbe \rm s}\le))$.
\end{corollary}
\begin{proof} The existence of $\Re_{\le R^{\e\prime}\be/\be R}(G^{\e\prime}\le)$, as well as the fact that it is a N\'eron model of $\Re_{\le K^{\prime}\!/\be K}\lbe(G^{\e\prime}_{\lbe \eta}\lbe)$, follow from \cite[\S10.1, proof of Proposition 4, p.~291]{blr}. The corollary is now immediate from the theorem. 
\end{proof}

\begin{remark}\label{rem} In the case where $R^{\e\prime}\be/R$ is a finite extension of discrete valuation rings, the isomorphism of Corollary \ref{cor1} was discussed in \cite[Proposition 3.19]{lor2} when the associated extension of fraction fields $K^{\le\prime}\be/K$ is Galois, $k$ is perfect and $G^{\e\prime}$ is the N\'eron model of an abelian variety over $K^{\prime}$. Corollary \ref{cor2} is proved in \cite[Proposition 1.1]{bb} (see also \cite[proof of Theorem 1]{ell}) when $K^{\le\prime}\be/K$ is separable, $G^{\e\prime}$ is as above and $R$ is henselian. The proofs of both results make use of a filtration introduced by B. Edixhoven.
\end{remark}

Let $R^{\e\prime}\be/R$ be a finite extension of discrete valuation rings with associated fraction field extension $K^{\le\prime}\be/K$ and residue field extension $k^{\e\prime}\be/k$. Let  $\kbar$ be a fixed algebraic closure of $k$ containing $k^{\e\prime}$ and let $\ksep$ and  $(k^{\e\prime})^{\rm s}$ denote, respectively, the separable closures of $k$ and $k^{\e\prime}$ inside $\kbar$. Now let $A_{K^{\le\prime}}$ be an abelian variety over $K^{\le\prime}$ with N\'eron model $A$ over $R^{\e\prime}$ and recall from \cite[IX, (1.2.2)]{sga7} Grothendieck's pairing
\begin{equation}\label{gpair}
\pi_{0}(\lbe A_{\le\textrm{s}}\lbe )((k^{\e\prime})^{\rm s})\times\pi_{0}(\lbe A^{\vee}_{\le\textrm{s}})((k^{\e\prime})^{\rm s})\to \Q/\Z\e,
\end{equation}
where $A^{\vee}$ is the N\'eron model of the abelian variety $A^{\vee}_{K^{\le\prime}}$ dual to $A_{K^{\le\prime}}$. 
When the preceding pairing is perfect, $\pi_{0}(\lbe A_{\le\textrm{s}}\lbe )((k^{\e\prime})^{\rm s})$ and $\pi_{0}(\lbe A^{\vee}_{\le\textrm{s}})((k^{\e\prime})^{\rm s})$ are isomorphic as abelian groups.  As noted by D. Lorenzini, it is not known, in general, whether $\pi_{0}(\lbe A_{\le\textrm{s}}\lbe )((k^{\e\prime})^{\rm s})$ and $\pi_{0}(\lbe A^{\vee}_{\le\textrm{s}})((k^{\e\prime})^{\rm s})$ are always isomorphic as abelian groups, i.e., even when \eqref{gpair} is not perfect. See \cite[p.~1, foonote]{lor} and \cite[p.~ 2032]{lor2}. Regarding this question, we observe the following consequence of  the previous corollary.

\begin{corollary}\label{cor3} Let $R^{\e\prime}\be/R$ be a finite extension of discrete valuation rings with associated fraction field extension $K^{\e\prime}\be/K$. Let $A_{K^{\prime}}$ be an abelian variety over $K^{\e\prime}$ with dual abelian variety $A^{\vee}_{K^{\le\prime}}$ and let $A$ and $A^{\vee}$ denote, respectively, the N\'eron models of $A_{K^{\prime}}$ and $A^{\vee}_{K^{\le\prime}}$ over $R^{\e\prime}$. Assume that Grothendieck's pairing \eqref{gpair} is perfect. Then there exists an isomorphism of abelian groups
\[
\pi_{0}(\Re_{\le R^{\e\prime}\be/\be R}(A\le)_{{\rm{s}}})(\ksep)\simeq \pi_{0}(\Re_{\le R^{\e\prime}\be/\be R}(A^{\vee}\le)_{{\rm{s}}}\lbe)(\ksep).
\]
\end{corollary}
\begin{proof} By Corollary \ref{cor1}, we have 
\[
\pi_{0}(\Re_{\le R^{\e\prime}\be/\be R}(A\le)_{{\rm{s}}})(\ksep)\simeq\Re_{\le k^{\le\prime}\be/k}(\pi_{0}(A_{\le\rm s}\le))(\ksep)=\pi_{0}(A_{\le \rm s}\le)(k^{\le\prime}\be\otimes_{k}\be\ksep).
\]
Now, by
\cite[V, \S6, no.~7, Theorem 4, p.~A.V.34]{bou} and an inductive limit argument, $k^{\le\prime}\otimes_{k}\ksep$ is a finite reduced $\ksep$-algebra. Consequently (see, e.g., \cite[Theorem 8.7, p.~90]{am}), $k^{\le\prime}\otimes_{k}\ksep\simeq\prod_{i\in I}k_{i}$ for some finite set $I$, where each $k_{i}$ is a finite field extension of $\ksep$. Thus 
$\pi_{0}(A_{\le \rm s}\le)(k^{\le\prime}\lbe\otimes_{k}\lbe\ksep)\simeq\prod_{i\in I} 
\pi_{0}(A_{\le \rm s}\le)(k_{i})$. Now, since each $k_{i}/\ksep$ is purely inseparable, the associated morphism $\spec k_{i}\to\spec \ksep$ is a universal homeomorphism. Thus, by \cite[VIII, Corollary 1.2]{sga4}, the canonical homomorphism $\pi_{0}(A_{\le \rm s}\le)(\ksep)\to \pi_{0}(A_{\le \rm s}\le)(k_{i})$ is an isomorphism for each $i\in I$. Further, since $(k^{\e\prime})^{\rm s}/\ksep$ is  also purely inseparable, the canonical homomorphism $\pi_{0}(A_{\le \rm s}\le)(\ksep)\to \pi_{0}(A_{\le \rm s}\le)((k^{\e\prime})^{\rm s})$ is an isomorphism as well. We conclude that $\pi_{0}(A_{\le \rm s}\le)(k^{\le\prime}\!\otimes_{k}\!\ksep)\simeq \pi_{0}(A_{\le \rm s}\le)((k^{\e\prime})^{\rm s})^{I}$. Similarly, $\pi_{0}(A^{\vee}_{\le \rm s}\le)(k^{\le\prime}\!\otimes_{k}\!\ksep)\simeq \pi_{0}(A^{\vee}_{\le \rm s}\le)((k^{\e\prime})^{\rm s})^{I}$. Since $\pi_{0}(\lbe A_{\le\textrm{s}}\lbe )((k^{\e\prime})^{\rm s})$ and $\pi_{0}(\lbe A^{\vee}_{\le\textrm{s}})((k^{\e\prime})^{\rm s})$ are isomorphic by the perfectness of \eqref{gpair}, the corollary follows.
\end{proof}

The interest of the above corollary is that, when $K^{\le\prime}\!/K$ is separable so that $\Re_{\le K^{\prime}\!/\lbe K}(A_{K^{\prime}}\be)$ is an abelian variety (see \cite[Proposition 4.1]{ed}), Grothendieck's pairing
\begin{equation}\label{gpair2}
\pi_{0}(\Re_{\le R^{\e\prime}\be/\lbe R}(A\le)_{{\rm{s}}})(\ksep)\times\pi_{0}(\Re_{\le R^{\e\prime}\be/\lbe R}(A^{\vee}\le)_{{\rm{s}}}\lbe)(\ksep)\to\Q/\Z
\end{equation}
may not be perfect \cite[comment after Corollary 2.2]{bb}, although, by the corollary, the groups involved are isomorphic when \eqref{gpair} is perfect. In other words, Corollary \ref{cor3} suggests that the answer to Lorenzini's question may well be positive.


\begin{thebibliography}{00}
\bibitem{am} M. F. Atiyah and  I. G. Macdonald,
 \textit{Introduction to Commutative Algebra}, Addison-Wesley Publishing Co., 
Reading, MA,  1969.

\bibitem{bb} A. Bertapelle and S. Bosch, 
\textit{Weil restriction and Grothendieck's duality conjecture},
J. Algebraic Geom.  {\bf 9} (2000), 155--164.

\bibitem{bga} A. Bertapelle  and C. D. Gonz\'alez-Avil\'es,
\textit{The Greenberg functor revisited.} In preparation. Preliminary version available at
\url{http://arxiv.org/abs/1311.0051v3}.


\bibitem{blr}  S. Bosch,  W.  L\"utkebohmert,  and M. Raynaud,
\textit{N\'eron models}, Ergebnisse der Math.  {\bf 21}, Springer-Verlag, Berlin, 1990.

\bibitem{bou} N. Bourbaki,
\textit{Algebra II}, Chapters 4--7, Masson, Paris, 1990. ISBN 0-387-19375-8.

 \bibitem{cgp} B. Conrad, O. Gabber and G. Prasad,
\textit{Pseudo-reductive Groups},
New Math. Monogr. {\bf 17}, Cambridge University Press, Cambridge, 2010.

\bibitem{sga3}  M. Demazure and   A. Grothendieck  (Eds.).
\textit{Sch\'emas en groupes. S\'eminaire de G\'eom\'etrie Alg\'ebrique du Bois Marie 1962-64 (SGA 3)}, augmented and corrected 2008--2011 re-edition of the original by P. Gille and P. Polo. Available at \url{http://www.math.jussieu.fr/~polo/SGA3}. 

\bibitem{ed} B. Edixhoven, 
 \textit{N\'eron models and tame ramification}, 
 Compos. Math.  {\bf 81} (1992), 291--306. 

\bibitem{ell} B. Edixhoven, Q.  Liu and D. Lorenzini,
 \textit{The $p$-part of the group of components of a N\'eron model}, J. Algebraic Geom.  {\bf 5} (1996), 801--813.


\bibitem{ega1} A. Grothendieck and J. Dieudonn\'e,
\textit{\'El\'ements de g\'eom\'etrie alg\'ebrique I. Le langage des sch\'emas},
Grundlehren  Math. Wiss. {\bf 166}, Springer-Verlag, Berlin, 1971.


\bibitem{ega} A. Grothendieck and  J. Dieudonn\'e, 
\textit{\'El\'ements de g\'eom\'etrie alg\'ebrique},
Publ. Math. Inst. Hautes \'Etudes Sci. 
 {\bf 8}  ($=\text{EGA II}$) (1961),  {\bf 24} ($=\text{EGA IV}_{2}$) (1965), {\bf 32}  ($=\text{EGA IV}_{4}$) (1967).

\bibitem{sga4}  A. Grothendieck et al.,
\textit{Th\'eorie des topos et cohomologie \'etale des sch\'emas (SGA 4-2). S\'eminaire de G\'eom\'etrie Alg\'ebrique du Bois Marie 1963--64},
 Lecture Notes in Math.  {\bf 270}, Springer-Verlag, Berlin,1972.

\bibitem{sga7}  A. Grothendieck et al.,
\textit{Groupes de monodromie en g\'eom\'etrie alg\'ebrique (SGA 7-1). S\'eminaire de G\'eom\'etrie Alg\'ebrique du Bois Marie 1967--69},
Lecture Notes in Math. {\bf 288}, Springer-Verlag, Berlin, 1972. 

\bibitem{lor} D. Lorenzini,
 \textit{Grothendieck's pairing for Jacobians and base change},
 J. Number Theory  {\bf 128} (2008), 1448--1457.

\bibitem{lor2} D. Lorenzini, 
 \textit{Torsion and Tamagawa numbers}, 
 Ann. Inst. Fourier   {\bf 61}, no. 5 (2011), 1995--2037.

\bibitem{mat} H.  Matsumura,  
\textit{Commutative Algebra},
 Second edition.  Math. Lecture Note Ser. {\bf 56}.  The Benjamin/Cummings Publishing Co., Inc., Reading, Mass., 1980.  ISBN 0-8053-7024-2.

\bibitem{s} T. Suzuki, \textit{ Grothendieck's pairing on Neron component groups: Galois descent from the semistable case}, arXiv:1410.3046v2.



\end{thebibliography}
\end{document}